\documentclass[UKenglish]{lmcs}
\pdfoutput=1

\usepackage{lastpage}
\lmcsdoi{16}{1}{21}
\lmcsheading{}{\pageref{LastPage}}{}{}%
{May~20,~2019}{Feb.~20,~2020}{}

\theoremstyle{plain}
\newtheorem{theorem}[thm]{Theorem}

\newtheorem{proposition}[thm]{Proposition}
\newtheorem{corollary}[thm]{Corollary}
\newtheorem{lemma}[thm]{Lemma}

\theoremstyle{definition}
\newtheorem{definition}[thm]{Definition}
\newtheorem{remark}[thm]{Remark}
\newtheorem{expl}[thm]{Example}

\usepackage{microtype}
\usepackage{graphicx,xcolor,cite}
\usepackage[utf8]{inputenc}
\usepackage{amsmath, amsxtra, amsfonts, amssymb, mathtools}
\usepackage{xspace, varioref}
\usepackage{tikz}
\usetikzlibrary{calc,positioning,patterns,snakes}
\usepackage{hyperref}

\tikzset{different/.style={dashed, line width=0.2mm}}

\newcommand{\systemAxis}[4]{
   \draw (#1-0.2,#2) -- (#1+#3,#2);
   \draw (#1,#2-0.2) -- (#1,#2+#4);
}
\newcommand{\mynode}[4]{ 
  \draw (#1,#2) circle (0.09) node (#3) {} #4;
}
\newcommand{\myhposition}[4]{ 
  \draw (#1,#2) -- (#1,#2-0.2) node (#3) {} #4;
}

\newcommand{\myvposition}[4]{ 
  \draw (#1,#2) -- (#1-0.2,#2) node (#3) {} #4;
}

\newcommand{\bC}{\mathfrak{C}}
\newcommand{\bB}{\mathfrak{B}}
\newcommand{\bA}{\mathfrak{A}}

\DeclareMathOperator{\mi}{mi}
\DeclareMathOperator{\Th}{Th}

\DeclareMathOperator{\Aut}{Aut}

\DeclareMathOperator{\Pol}{Pol}

\DeclareMathOperator{\Csp}{CSP}
\DeclareMathOperator{\Age}{Age}

\DeclareMathOperator{\AND}{\wedge}
\DeclareMathOperator{\OR}{\vee}

\newcommand{\ignore}[1]{}
\newcommand{\cK}{\mathcal{K}}
\newcommand{\Q}{\mathbb{Q}}
\newcommand{\N}{\mathbb{N}}
\newcommand{\set}[1]{\left\{#1\right\}}
\newcommand{\mybar}[1]{#1}
\newcommand{\ceq}{\mathrel{\vcentcolon =}}  
\newcommand{\cequiv}{\mathrel{\vcentcolon\equiv}}

\newcommand{\Fresse}{Fra\"{i}ss\'{e}}

\begin{document}

\title[The Complexity of Combinations of Qualitative CSPs]{The Complexity of Combinations of Qualitative Constraint Satisfaction Problems}
\titlecomment{Both authors have received funding from the European Research Council (ERC  Grant Agreement no. 681988), the German Research Foundation (DFG, project number 622397), and the DFG Graduiertenkolleg 1763 (QuantLA). In contrast to the conference version~\cite{BodirskyGreiner18}, this is the full article with all proofs and a stronger main result (Theorem~\ref{thm:r2}).}

\author{Manuel Bodirsky}
\address{Institut f\"ur Algebra, Technische Universit\"at Dresden, Germany}

\author{Johannes Greiner}




\begin{abstract}
The CSP of a first-order theory $T$ is 
the problem of deciding for a given 
finite set $S$ of atomic formulas whether $T \cup S$ is satisfiable. 
Let $T_1$ and $T_2$ be two theories 
with countably infinite models and disjoint signatures. 
Nelson and Oppen presented conditions that imply 
decidability (or polynomial-time decidability) of 
$\Csp(T_1 \cup T_2)$ 
under the assumption that $\Csp(T_1)$ 
and $\Csp(T_2)$ are decidable (or polynomial-time decidable). 
We show that for a large class of $\omega$-categorical theories $T_1, T_2$ the Nelson-Oppen conditions are not only sufficient, but also necessary for polynomial-time tractability of $\Csp(T_1 \cup T_2)$ (unless P=NP).  
\end{abstract}

\maketitle

\section{Introduction}
Two independent proofs of the finite-domain constraint satisfaction tractability conjecture have recently been published by Bulatov and Zhuk~\cite{BulatovFVConjecture,ZhukFVConjecture}, settling the Feder-Vardi dichotomy conjecture.  In contrast, the computational complexity of constraint satisfaction problems over infinite domains cannot be classified in general~\cite{BodirskyGrohe}.
However, for a restricted class  of constraint satisfaction problems that strictly contains  all finite-domain CSPs and captures the vast majority of the problems studied in \emph{qualitative reasoning} (see the survey article~\cite{Qualitative-Survey}) there also is a tractability conjecture (see~\cite{BPP-projective-homomorphisms,wonderland,BartoPinskerDichotomy,BKOPP}).
The situation is similar to the situation 
for finite-domain CSPs before Bulatov and Zhuk: there is a formal condition which provably implies NP-hardness, and the conjecture is that every other CSP in the class is in P. 

For finite domain CSPs, it turned out that only few fundamentally different algorithms were needed to complete the classification; the key in both the solution of Bulatov and the solution of Zhuk
was a clever combination of the existing algorithmic ideas. An intensively studied method for obtaining (polynomial-time) decision procedures for infinite-domain CSPs is the Nelson-Oppen combination method; see, e.g.,~\cite{BaaderSchulz, NelsonOppen79}.
The method did not play any role for the classification of finite-domain CSPs, but is extremely powerful for combining algorithms for infinite-domain CSPs. 

In order to conveniently state what type of combinations of CSPs can be studied with the Nelson-Oppen method, we slightly generalise
the notion of a CSP. The classical definition is to fix an infinite structure $\bB$ with finite relational signature $\tau$; then $\Csp(\bB)$ is the computational problem of deciding whether a given finite set of \emph{atomic $\tau$-formulas} (i.e., formulas of the form $x_1=x_2$ or of the form $R(x_1,\dots,x_n)$ for $R \in \tau$ and variables $x_1,\dots,x_n$) is satisfiable in $\bB$. 
Instead of fixing a $\tau$-structure $\bB$, we fix a \emph{$\tau$-theory $T$} (i.e., 
a set of first-order $\tau$-sentences). 
Then $\Csp(T)$ is the computational problem of deciding for a given finite set $S$ of atomic $\tau$-formulas whether 
$T \cup S$ has a model. Clearly, this is a generalisation of the classical definition since $\Csp(\bB)$ is the same as $\Csp(\Th(\bB))$ where $\Th(\bB)$ is the \emph{first-order theory of $\bB$}, i.e., the set of all first-order sentences that hold in $\bB$.
The definition for theories is \emph{strictly} more expressive (we give an example in Section~\ref{sect:gen-comb} that shows this). 

\ignore{Let $T$ be a first-order theory
over a finite relational signature $\tau$. 
Then the problem $\Csp(T)$ is the computational problem of deciding 
 for a given finite set of atomic 
 $\tau$-formulas
 $\{\psi_1,\dots,\psi_m\}$ whether $T \cup \{\exists \bar x (\psi_1 \wedge \cdots \wedge \psi_m)\}$ is satisfiable.}
 
Let $T_1$ and $T_2$ be two theories 
with disjoint finite relational signatures
$\tau_1$ and $\tau_2$. 
We are interested in the question when
$\Csp(T_1 \cup T_2)$ can be solved in polynomial time; we refer to this problem as
the \emph{combined CSP} for $T_1$ and $T_2$. Clearly, if $\Csp(T_1)$
or $\Csp(T_2)$ 
is NP-hard, then $\Csp(T_1 \cup T_2)$ is NP-hard, too. Suppose now that $\Csp(T_1)$ 
and $\Csp(T_2)$ 
can be solved in polynomial-time. 
In this case, there are examples where $\Csp(T_1 \cup T_2)$ is in P, and
examples where $\Csp(T_1 \cup T_2)$ is NP-hard. Even if we know the complexity of $\Csp(T_1)$ and of $\Csp(T_2)$, a classification of the complexity of $\Csp(T_1 \cup T_2)$ for arbitrary 
theories $T_1$ and $T_2$ is too ambitious (see Section~\ref{sect:non-class} for a formal justification). 
But such a classification 
should be feasible
at least for the mentioned class of infinite-domain CSPs for which the tractability conjecture applies.

\subsection{Qualitative CSPs} 
\label{sect:qcsps}
The idea of \emph{qualitative formalisms} is
that reasoning tasks (e.g., about space and time) is not performed with absolute numerical values, but rather with \emph{qualitative} predicates (such as \emph{within}, \emph{before}, etc.).
There is no universally accepted definition in the literature that defines what a \emph{qualitative CSP} is, but a proposal has been made in~\cite{Qualitative-Survey}; the central mathematical property   
for this proposal is \emph{$\omega$-categoricity}. 
A theory is called \emph{$\omega$-categorical} 
if it has up to isomorphism only one countable model. 
A structure is called \emph{$\omega$-categorical} if and only if its first-order theory is $\omega$-categorical. 
Examples are $({\mathbb Q};<)$, Allen's Interval Algebra, and more generally all homogeneous
structures with a finite relational
signature
(a structure $\bB$ is called \emph{homogeneous} if all isomorphisms between finite  substructures 
can be extended to an automorphism; see~\cite{HodgesLong,Bodirsky-HDR-v8}). 
The class of CSPs
for $\omega$-categorical theories 
arguably coincides with the class of CSPs for \emph{qualitative formalisms}
studied e.g.\ in temporal and spatial reasoning; see~\cite{Qualitative-Survey}. 

For an $\omega$-categorical theory $T$,
the complexity of $\Csp(T)$  can be studied using the universal-algebraic approach that led to the proof of the Feder-Vardi dichotomy conjecture. A \emph{signature} is a set of function and relation symbols together with an arity for each symbol. Let $\tau$ be a signature. A \emph{$\tau$-structure} $\bB$ is a set $B$ together with an $n$-ary function on $B$ for each $n$-ary function symbol in $\tau$ and an $n$-ary relation over $B$ for each $n$-ary relation symbol in $\tau$.
  All signatures of structures in this article are assumed to be countable.
One of the central concepts for the universal-algebraic approach is the concept of a \emph{polymorphism} of a structure $\bB$,
i.e., a homomorphism from $\bB^k$ to $\bB$ for some $k \in {\mathbb N}$. 
It is known that the polymorphisms of a finite structure $\bB$ fully capture the complexity of $\Csp(\bB)$ up to P-time reductions (in fact, up to Log-space reductions; see~\cite{dagstuhlCSP18} for a collection of survey articles about the complexity of CSPs), and the same is true for structures $\bB$ with an $\omega$-categorical theory~\cite{BodirskyNesetrilJLC}. 
In order to understand when we can apply the universal-algebraic approach to study the complexity of $\Csp(T_1 \cup T_2)$, we need to understand the following fundamental question.

\medskip
{\bf Question 1:} Suppose that $T_1$ and $T_2$ are 
theories with disjoint finite relational signatures $\tau_1$ and $\tau_2$. When is there an $\omega$-categorical $(\tau_1 \cup \tau_2)$-theory $T$ 
such that $\Csp(T)$ equals $\Csp(T_1 \cup T_2)$, i.e., for all sets $S$ of atomic $(\tau_1 \cup \tau_2)$-formulas,
we have that $S \cup T$ is satisfiable if and only if $S \cup (T_1 \cup T_2)$ is satisfiable? 

\medskip 
Note that $\omega$-categorical theories are \emph{complete}, i.e., for every first-order sentence $\phi$ either $T$ implies $\phi$ 
or $T$ implies $\neg \phi$. 
In general, it is not true that $\Csp(T_1 \cup T_2)$ equals $\Csp(T)$ for a complete theory $T$ (we present an example in Section~\ref{sect:gen-comb}). 

Question 1 appears to be very difficult in general, in particular when considering the work of Braunfeld~\cite{Braunfeld19} in the context of Proposition~\ref{prop:csp}.
However, we present
a broadly applicable condition for $\omega$-categorical theories 
$T_1$ and $T_2$ with infinite models that implies the existence of an $\omega$-categorical theory $T$ such that $\Csp(T_1 \cup T_2)$ equals $\Csp(T)$ (Proposition~\ref{prop:r1} below). 
The theory $T$ that we construct has many useful properties. In particular $T_1 \cup T_2 \subseteq T$ and 

\begin{enumerate}
\item if $\phi_1(\bar x)$ is a $\tau_1$-formula 
and $\phi_2(\bar x)$ is a $\tau_2$-formula, both with free variables $\bar x = (x_1,\dots,x_n)$, then 
$T \models \exists \bar x (\phi_1(\bar x) \wedge \phi_2(\bar x) \wedge \bigwedge_{i<j} x_i \neq x_j)$ if and only if
$T_1 \models  \exists \bar x (\phi_1(\bar x) \wedge \bigwedge_{i<j} x_i \neq x_j)$ and $T_2 \models  \exists \bar x (\phi_2(\bar x) \wedge \bigwedge_{i<j} x_i \neq x_j)$;  \label{item:satisfiableTogether}
\item For every $(\tau_1 \cup \tau_2)$-formula $\phi$ there exists a conjunction of $\tau_1$ and $\tau_2$ formulas that is equivalent to $\phi$ modulo $T$.\label{item:canSeperateFormula}
\end{enumerate}
In fact, $T$ is uniquely given by these three properties (up to equivalence of theories; see Lemma~\ref{lem:unique}) and again $\omega$-categorical, and we call it the \emph{generic combination of $T_1$ and $T_2$}. 
Let $\bB_1$ and $\bB_2$ be two $\omega$-categorical relational structures whose first-order theories have a generic combination $T$; then we call the (up to isomorphism unique) countably infinite model of $T$ the
\emph{generic combination of  $\bB_1$ and $\bB_2$}.

\subsection{The Nelson-Oppen Criterion}
Let $T_1,T_2$ be theories with disjoint finite signatures $\tau_1,\tau_2$ and suppose that $\Csp(T_1)$ is in P and $\Csp(T_2)$ is in P. 
Nelson and Oppen gave sufficient conditions
for $\Csp(T_1 \cup T_2)$ to be solvable
in polynomial time, too. Their conditions are: 
\begin{enumerate}
\item Both $T_1$ and $T_2$ are \emph{stably infinite}:
a $\tau$-theory $T$ is called \emph{stably infinite}
if for every quantifier-free $\tau$-formula $\phi(x_1,\dots,x_n)$,
if $\phi$ is satisfiable over $T$, then
there also exists an \emph{infinite} model $\bA$
and elements $a_1,\dots,a_n$ such that 
$\bA \models \phi(a_1,\dots,a_n)$. 
\item for $i=1$ and $i=2$, 
the signature $\tau_i$ contains a binary relation symbol $\neq_i$ that denotes the disequality relation, i.e., $T_i$ implies the sentence 
$\forall x,y \, (x \neq_i y \Leftrightarrow \neg(x=y))$; 
\item Both $T_1$ and $T_2$ are \emph{convex} (here we follow established terminology). 
A $\tau$-theory $T$ is called \emph{convex} if for every finite set $S$ of atomic $\tau$-formulas
the set 
$T \cup S \cup \{x_1 \neq y_1,\dots,x_m \neq y_m\}$ is satisfiable whenever 
$T \cup S \cup \{x_j \neq y_j\}$ is satisfiable for each $j \leq m$. 
\end{enumerate}
The assumption that a relation symbol denoting the disequality relation is part of the signatures $\tau_1$ and $\tau_2$ is often implicit in the literature treating the Nelson-Oppen method. It would be interesting to explore when it can be dropped, but we will not pursue this question here. 
The central question of this article is the following. 

\medskip

{\bf Question 2.} In which settings are the Nelson-Oppen
conditions (and in particular, the convexity condition) not only sufficient, but also necessary for polynomial-time tractability of the combined CSP?

\medskip 

Again, for general
theories $T_1$ and $T_2$, this is a too ambitious research goal; 
but we will study it for generic combinations
of $\omega$-categorical theories $T_1,T_2$ with infinite models. 
In this setting, 
the first condition that both $T_1$ and $T_2$ are stably infinite is trivially satisfied. 
The third condition on $T_i$, 
convexity, is 
equivalent to the existence of a
binary injective polymorphism of the
(up to isomorphism unique) countably infinite model of $T_i$ (see Theorem~\ref{thm:hdr}). We mention that binary injective polymorphisms played an important role in several recent infinite-domain complexity classifications~\cite{ecsps,Phylo-Complexity,posetCSP16}. 

\subsection{Results} 
To state our results concerning Question 1 and Question 2 we need basic terminology for permutation groups.
A permutation group $G$ on a set $A$ is called 
\begin{itemize}
\item \emph{$n$-transitive} if for all tuples $\bar b, \bar c \in A^n$, each with pairwise distinct entries, there exists a permutation $g \in G$ such that $g(\bar b) = \bar c$ (unary functions applied to tuples or multiple arguments act componentwise and result in tuples).
$G$ is called \emph{transitive} if it is 1-transitive. 
\item \emph{$n$-set-transitive} 
if for all subsets $B,C$ of $A$ with $|B|=|C|=n$ there exists a permutation $g \in G$ such that $g(B) \ceq \{g(b) \mid b \in B\} = C$. 
\end{itemize}
A structure is called \emph{$n$-transitive} (or \emph{$n$-set-transitive}) if its automorphism group is. 
The existence of
generic combinations can be characterised as follows (see Section~\ref{sect:gen-comb} for the proof).

\begin{proposition}\label{prop:r1}
Let $\bB_1$ and $\bB_2$ be countably infinite $\omega$-categorical structures with disjoint relational signatures. Then $\bB_1$ and $\bB_2$ have a generic combination if and only if either both $\bB_1$ and $\bB_2$ do not have algebraicity  (in the model-theoretic sense; see Section~\ref{sect:gen-comb})
or one of $\bB_1$ and $\bB_2$ does have algebraicity and the other has an automorphism group which is $n$-transitive for all  $n \in {\mathbb N}$.
\end{proposition}

Our main result concerns Question 2
for generic combinations $\bB$ of countably infinite
$\omega$-categorical structures $\bB_1$ and $\bB_2$;
as we mentioned before, if the generic combination exists, it is up to isomorphism unique, and again $\omega$-categorical.


Our result does not apply if one of the structures is too inexpressive. To state the expressivity requirement we introduce further terminology. A \emph{primitive positive} formula (\emph{pp-formula}) is a conjunction of atomic formulas where some variables can be existentially quantified. Likewise, a relation $R \subseteq B^n$ is \emph{pp-definable} in a structure $\bB$ if there exists a pp-formula $\phi$ such that $R = \set{\bar{x}\in B^n \mid \bB \models \phi(\bar{x})}$.
We call a tuple \emph{injective} if all its coordinates are pairwise distinct.

For a fixed $\tau$-structure $\bB$ and a $\tau$-formula $\phi(x_1, \dotsc, x_n)$ a map $s \colon \set{x_1, \dotsc, x_n} \rightarrow \bB$ such that $\phi(s(x_1),\dotsc, s(x_n))$ is true in $\bB$ is a \emph{solution} to $\phi$. 
  The expressivity requirement for our result is captured in the following notion.

\begin{definition}\label{def:preventCrosses}
  A structure $\bB$ \emph{can prevent crosses} if there exists a pp-formula $\phi$ such that
  \begin{enumerate}
  \item $\phi(x,y,u,v) \AND x=y$ has a solution $s$ over $\bB$ such that $s(x,u,v)$ is injective,
  \item $\phi(x,y,u,v) \AND u=v$ has a solution $s$ over $\bB$ such that $s(x,y,u)$ is injective,
  \item $\phi(x,y,u,v) \AND x=y \AND u=v$ has no solution over $\bB$.
  \end{enumerate}
  Any such formula $\phi$ will be referred to as a \emph{cross prevention formula} of $\bB$.
\end{definition}

To explain the naming, consider solutions  $s_1,\,s_2$  to a cross prevention formula $\phi$ of a structure $\bB$ with images $(x_1,x_1,u_1,v_1)$ and $(x_2,y_2,u_2,u_2)$ respectively, such that $(x_1,u_1,v_1)$ and $(x_2,y_2,u_2)$ are injective.
Let $f\colon \bB^k \rightarrow \bB$ be a polymorphism of $\bB$. If $R$ is a relation in $\bB$ and $t_1, \dotsc, t_k\in R$, then $f(t_1, \dotsc, t_k) \in R$ because $f$ is a homomorphism ($f$ \emph{preserves} $R$). Preservation of relations transfers to any relation which is pp-definable over $\bB$.
Therefore, there is no binary polymorphism $f$ of $\bB$ such that
  \begin{align*}
    f(x_1,x_2) &=  f(x_1,y_2) \quad \text{and}\\
    f(u_1,u_2) &= f(v_1,u_2),
  \end{align*}
because the arguments of $f$ are in the relation defined by $\phi$, but the image is not.
This means that in Figure~\ref{fig:preventCrosses} either the vertical or the horizontal identity (indicated by dashed lines) cannot hold, i.e., we will not see a cross.

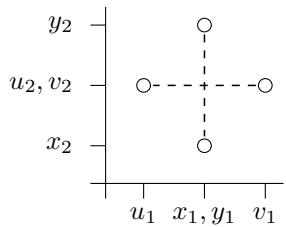
\begin{figure}
  \begin{center}
  \begin{tikzpicture}
    \def\xbetweenspace{0.8};
    \def\ybetweenspace{0.8};
    \def\belowspace{0};
    \def\leftspace{0.1};
    \def\xoffset{0.5};
    \def\yoffset{0.5};
    
    \myhposition{\xoffset+0*\xbetweenspace}{0}{1}{node[align=left, below=\belowspace] {$u_1$}};
    \myhposition{\xoffset+1*\xbetweenspace}{0}{2}{node[align=left, below=\belowspace] {$x_1, y_1$}};
    \myhposition{\xoffset+2*\xbetweenspace}{0}{3}{node[align=left, below=\belowspace] {$v_1$}};
%
    \myvposition{0}{\yoffset+0*\ybetweenspace}{4}{node[align=right, left=\leftspace] {$x_2$}};
    \myvposition{0}{\yoffset+1*\ybetweenspace}{5}{node[align=right, left=\leftspace] {$u_2,v_2$}};
    \myvposition{0}{\yoffset+2*\ybetweenspace}{6}{node[align=right, left=\leftspace] {$y_2$}};

    \mynode{\xoffset+1*\xbetweenspace}{\yoffset+0*\ybetweenspace}{x}{node[align=right, below=\leftspace] {}};
    \mynode{\xoffset+2*\xbetweenspace}{\yoffset+1*\ybetweenspace}{u}{node[align=right, below=\leftspace] {}};
    \mynode{\xoffset+0*\xbetweenspace}{\yoffset+1*\ybetweenspace}{v}{node[align=right, left=\leftspace] {}};
    \mynode{\xoffset+1*\xbetweenspace}{\yoffset+2*\ybetweenspace}{y}{node[align=right, left=\leftspace] {}};

    \draw[different] (y) -- (x);
    \draw[different] (u) -- (v);

    \systemAxis{0.0}{0.0}{\xoffset+2.3*\xbetweenspace}{\yoffset+2.3*\ybetweenspace}; 
    
  \end{tikzpicture}
\end{center}
\caption{An illustration of the property to \emph{prevent crosses} from Definition~\protect\ref{def:preventCrosses}. Dashed edges indicate (potential) equalities between function values.}
\label{fig:preventCrosses}
\end{figure}

To simplify the presentation of our result and its proof, we introduce the following shortcut.
\begin{definition}\label{def:J}\hypertarget{def:J}{}
A countably infinite $\omega$-categorical structure has Property~\emph{$J$} if it 
\begin{itemize}
\item has a relational signature containing a binary symbol for the disequality relation,
\item is 2-set-transitive,
\item can prevent crosses and
\item does not have algebraicity.
\end{itemize}
\end{definition}
\noindent An easy example for such a structure is $(\Q; <, \neq)$. Further examples of structures with \hyperlink{def:J}{Property~$J$} come from
expansions of the countable random tournament (see, e.g., Lachlan~\cite{Lachlan}, see Section~\ref{sect:gen-comb} for the definition of expansion) and the countable homogeneous local order $S(2)$ (also see~\cite{Cherlin}). More examples for structures with \hyperlink{def:J}{Property~$J$} will be presented in Section~\ref{sec:examples}. 

We now state our main result for Question 2. The proof can be found in Section~\ref{sect:bin-inj}.

\begin{theorem}\label{thm:r2}\label{THM:R2}
Let $\bB$ be the generic combination of two countably infinite $\omega$-categorical structures $\bB_1$ and $\bB_2$ with \hyperlink{def:J}{Property~$J$} and finite relational signature such that $\Csp(\bB_1)$ and $\Csp(\bB_2)$ are in P. Then one of the following applies:
\begin{itemize}
\item $\Th(\bB_1)$ or $\Th(\bB_2)$ is not convex; in this case, $\Csp(\bB)$ is NP-hard. 
\item Each of $\Th(\bB_1)$ and $\Th(\bB_2)$ is convex, and $\Csp(\bB)$ is in P. 
\end{itemize}
\end{theorem}

\ignore{
\begin{theorem}
Let $\bB$ be the generic combination of two $\omega$-categorical structures $\bB_1$ and $\bB_2$ with \hyperlink{def:J}{Property~$J$} and such that $\Csp(\bB_1')$ and $\Csp(\bB_2')$ are in P for all finite signature reducts $\bB_1'$ of $\bB_1$ and $\bB_2'$ of $\bB_2$. Then one of the following applies:
\begin{itemize}
\item $\Th(\bB_1)$ or $\Th(\bB_2)$ is not convex; in this case, $\bB$ has a finite signature reduct $\bB'$ such that $\Csp(\bB')$ is NP-hard.
\item Each of $\Th(\bB_1)$ and $\Th(\bB_2)$ is convex, and for every finite signature reduct $\bB'$ of $\bB$  $\Csp(\bB')$ is in P.
\end{itemize}
\end{theorem}

  }

\noindent In other words, either the Nelson-Oppen conditions apply, and $\Csp(\bB)$ is in P, 
or otherwise $\Csp(\bB)$ is NP-complete.

Again, the easiest examples for structures satisfying the assumptions of Theorem~\ref{thm:r2} come from $(\mathbb Q,<)$ (see Section~\ref{sec:examples}). The CSPs for first-order reducts of $(\mathbb Q,<)$ (see Section~\ref{sect:gen-comb} for the definition of first-order reduct) have been called \emph{temporal CSPs} and their computational complexity has been classified~\cite{tcsps-journal}. 
There are many interesting polynomial-time tractable temporal CSPs that have non-convex theories, which makes
temporal CSPs a particularly interesting class for understanding the situation where the Nelson-Oppen conditions do not apply.
Generic combinations of temporal CSPs are isomorphic to first-order reducts of the \emph{countable random permutation} introduced in~\cite{CameronPermutations} and studied in~\cite{LinmanPinsker};
a complexity classification of the CSPs of all reducts of the random permutation (as e.g.\ in~\cite{tcsps-journal,Phylo-Complexity,posetCSP16} for simpler structures than the random permutation) is out of reach for the current methods. In principle, the method of canonical functions (see~\cite{BMPP16}) can be applied but leads to a huge amount of cases and seems combinatorially intractable at the moment.

\subsection{Outline of the article}
The rest of the article is organized as follows. In Section~\ref{sect:gen-comb} we introduce generic combinations, study existence and uniqueness and prove a result that will help us to prove satisfiability of conjunctions of formulas.
Afterwards, in Section~\ref{sect:bin-inj}, we prove the main result about the computational complexity of generic combinations of CSPs.
In Section~\ref{sec:examples} we present examples for structures with \hyperlink{def:J}{Property~$J$} and for Theorem~\ref{thm:r2}.
Finally, in Section~\ref{sect:non-class} we show that a general complexity classification for all generic combinations is not feasible.

\section{Generic Combinations}
\label{sect:gen-comb}
We already mentioned that our definition of CSPs for theories is a strict generalisation of
  the notion of CSPs for structures, 
  and this will be clarified by the following proposition which is an immediate
  consequence of Proposition~2.4.6 in~\cite{Bodirsky-HDR-v8}. 
   
  
  
\begin{proposition}\label{prop:csp}
Let $T$ be a first-order theory with finite relational signature. Then there exists a structure $\bB$ such that 
$\Csp(\bB) = \Csp(T)$ if and only if 
$T$ has the \emph{Joint Homomorphism Property (JHP)}, that is, for any two models $\bA$, $\bB$ of $T$ there exists a model $\bC$ of $T$ such that both $\bA$ and $\bB$ homomorphically map to $\bC$. 
\end{proposition}

\begin{expl}\label{expl:jhp}
A simple example of two theories $T_1,T_2$ with the JHP such that $T_1 \cup T_2$ does not have the JHP is given by
\begin{align*}
T_1  & \ceq \{ \forall x,y \; ((O(x) \wedge O(y)) \Rightarrow x=y ) \} \\
T_2  & \ceq \{ \forall x.  \; \neg (P(x) \wedge Q(x))  \}
\end{align*}
Suppose for contradiction that $T_1 \cup T_2$ has the JHP. 
Note that 
\begin{align*}
& T_1 \cup T_2 \cup \{\exists x (O(x) \wedge P(x))\} \quad
\text{and } \quad  T_1 \cup T_2 \cup \{\exists y (O(y) \wedge Q(y))\}
\end{align*}
are satisfiable. The JHP implies that 
$$T_1 \cup T_2 \cup \{\exists x (O(x) \wedge P(x)), \exists y (O(y) \wedge Q(y))\}$$ has a model
$\bA$, so $\bA$ has elements
$u,v$ satisfying 
$O(u) \wedge O(v) \wedge P(u) \wedge Q(v)$. 
Since $\bA \models T_1$ 
we must have $u=v$, 
and so $\bA$ does not satisfy the sentence 
$\forall x. \; \neg (P(x) \wedge Q(x))$ from $T_2$,
a contradiction.  
\end{expl}

For general theories $T_1, T_2$ even the question whether $T_1 \cup T_2$ 
has the JHP might be 
a difficult question. 
But if both $T_1$ and $T_2$ are $\omega$-categorical with a countably infinite model that  \emph{does not have algebraicity},
then $T_1 \cup T_2$ always has the JHP 
(a consequence of Lemma~\ref{lem:combine-2} below). 
A structure $\bB$ (and its first-order theory) \emph{does not have algebraicity} if for all first-order 
formulas $\phi(x_0,x_1,\dots,x_n)$
and all elements $a_1,\dots,a_n \in B$ the set $\{a_0 \in B \mid  \bB \models \phi(a_0,a_1,\dots,a_n)\}$
 is either infinite or contained in $\{a_1,\dots,a_n\}$; otherwise, we say that
the structure \emph{has algebraicity}. 


Next, we will introduce  \emph{strong amalgamation} which has a tight link to algebraicity.
The \emph{age} of a relational $\tau$-structure $\bB$ is the class of all finite $\tau$-structures that embed into $\bB$.  
A class $\cK$ of structures has the \emph{amalgamation property} if for all $\bA, \bB_1, \bB_2 \in \cK$ and embeddings $f_i \colon \bA \rightarrow \bB_i$, for $i=1$ and $i=2$, there exist $\bC \in \cK$ and embeddings $g_i \colon \bB_i \rightarrow \bC$ such that $g_1 \circ f_1 = g_2 \circ f_2$. It has the \emph{strong amalgamation property} if additionally $g_1(B_1) \cap g_2(B_2) = g_1(f_1(A)) = g_2(f_2(A))$.
If $\cK$ is a class of structures with relational signature which is closed under isomorphism, substructures, and has the amalgamation property, then there exists an  (up to isomorphism unique) 
countable homogeneous structure $\bB$ whose age is $\cK$. This structure $\bB$ is called \emph{\Fresse -limit} of $\cK$ (see~\cite{HodgesLong}, page 326).
Moreover, in this case, we have the following:
\begin{propC}[{\cite[p.~330]{HodgesLong}}]\label{prop:noAlgStrongAmg}
  If $\cK$ is the age of a countable structure $\bB$,
  then $\bB$ has no algebraicity if and only if $\cK$ has the strong amalgamation property.  
\end{propC}

The significance of strong amalgamation in the theory of
combining decision procedures
has already been pointed out by Bruttomesso, Ghilardi, and Ranise~\cite{BruttomessoGR14}. 

A structure $\bB_1$ is called a \emph{reduct} of a structure $\bB_2$, and $\bB_2$ is called an \emph{expansion} of $\bB_1$, if $\bB_1$ is obtained from $\bB_2$ by dropping some of the relations of $\bB_1$. 
If $\bB_1$ is a reduct of $\bB_2$ with the signature $\tau$ then we write $\bB_2^{\tau}$ for $\bB_1$.
An expansion $\bB_2$ of $\bB_1$ is called a \emph{first-order expansion} if all additional relations in $\bB_2$ have a first-order definition in $\bB_1$. 
A structure $\bB_1$ is called a \emph{first-order reduct} if $\bB_2$ is a reduct of a first-order expansion 
of $\bB_2$.
Note that if a structure $\bB$ is $2$-set-transitive then so is every first-order reduct of $\bB$ (since its automorphism group contains the automorphisms of $\bB$).

Let $\tau_1$ and $\tau_2$ be disjoint relational
signatures, and let ${\cK}_i$ be a class of isomorphism-closed finite $\tau_i$-structures, for $i \in \{1,2\}$. 
Then ${\cK}_1 * {\cK}_2$ denotes the class of $(\tau_1 \cup \tau_2)$-structures given by
$\{ \bA \mid \bA^{\tau_1} \in \cK_1 \text{ and } \bA^{\tau_2} \in \cK_2)\}$.

If $\bar{a} = (a_1,\dots,a_n) \in B^n$ and $G$ is a permutation group on $B$
then
\[
G \bar{a} \ceq \{(\alpha(a_1),\dots,\alpha(a_n)) \mid \alpha \in G\}
\]
is called the \emph{orbit} of $\bar a$ (with respect to $G$); orbits of pairs (i.e., $n=2$) are also called \emph{orbitals}. Orbitals of pairs of equal elements are called \emph{trivial}. In this article, whenever we have a fixed structure $\bB$ and not otherwise noted, $G$ will always be the automorphisms group of $\bB$.

The following theorem is one of the most important tools when dealing with $\omega$-categorical structures.
\begin{theorem}[Engeler, Ryll-Nardzewski, Svenonius, see~{\cite[p. 341]{HodgesLong}}]\label{thm:ryll}
  Let $\bB$ be a countably infinite structure with countable signature. Then, the following are equivalent:
  \begin{enumerate}
  \item $\bB$ is $\omega$-categorical;
  \item for all $n\geq 1$ every orbit of $n$-tuples is first-order definable in $\bB$;
  \item for all $n\geq 1$ there are only finitely many orbits of $n$-tuples.
  \end{enumerate}
\end{theorem}

\noindent Theorem~\ref{thm:ryll} implies that a homogeneous structure with finite relational signature is $\omega$-categorical, and the expansion of an $\omega$-categorical structure by all first-order definable relations (i.e., all orbits of its age) is homogeneous.
We are now ready to prove the first result about generic combinations.

\begin{lemma}\label{lem:combine-2}
Let $T_1$ and $T_2$ be $\omega$-categorical theories with disjoint relational signatures $\tau_1$ and $\tau_2$, with infinite models 
without algebraicity. Then there exists an
$\omega$-categorical model $\bB$
 of $T_1 \cup T_2$ without algebraicity 
such that 
\begin{align}
\text{ for all } k \in {\mathbb N}, \bar a,\bar b \in B^{k} \text{ injective} \colon  & \quad 
\Aut(\bB^{\tau_1}) \bar a \cap \Aut(\bB^{\tau_2}) \bar b \neq \emptyset  \quad\text{and}
\label{eq:free}
\\
  \text{ for all } k \in {\mathbb N}, \bar a \in B^{k} \text{ injective}  \colon &  \quad \Aut(\bB^{\tau_1}) \bar a \cap \Aut(\bB^{\tau_2}) \bar a  = \Aut(\bB) \bar a \, . 
\label{eq:splits}
\end{align}
\end{lemma}

\begin{proof}
Let $\bB_i$ be a countably infinite
model of $T_i$ without algebraicity for $i=1$ and $i=2$.  
Consider the $\sigma_i$-expansion $\bC_i$ of $\bB_i$ by all orbits of $\Aut(\bB_i)$,
and choose the signatures of $\bC_1$ and $\bC_2$ to be disjoint. Then, by definition, $\bC_i$ is homogeneous and $\omega$-categorical and Theorem~\ref{thm:ryll} implies that $\bC_i$ is first-order interdefinable with $\bB_i$ and thus without algebraicity. Hence, $\Age(\bC_i)$ is a strong amalgamation class and therefore, $\cK \ceq \Age(\bC_1) * \Age(\bC_2)$
 is a strong amalgamation class, too. 
 Let $\bC$ be the \Fresse -limit of $\cK$. By construction of $\bC$ two $n$-tuples $t$, $s$ are isomorphic over $\bC$ iff $t$ and $s$ satisfy the same $n$-ary relations from $\bC_1$ and $\bC_2$. Therefore, there are only finitely many non-isomorphic $n$-tuples over $\bC$, and hence, homogeneity of $\bC$, there are only finitely many orbits of $n$-tuples over $\bC$. This implies $\omega$-categoricity of $\bC$. By strong amalgamation of $\cK$ and Proposition~\ref{prop:noAlgStrongAmg}, $\bC$ has no algebraicity.

 



Hence, the
$(\tau_1 \cup \tau_2)$-reduct $\bB$ of $\bC$ 
is $\omega$-categorical and has no algebraicity.
To show $(\ref{eq:free})$, 
let $\bar a, \bar b \in B^{k}$ be injective tuples and
let $\bA$ be the
$(\sigma_1 \cup \sigma_2)$-structure with domain $A = \{1,\dots,k\}$ such that $i \mapsto a_i$
is an embedding of $\bA$ into $\bC^{\sigma_1}$
and $i \mapsto b_i$ is an 
embedding of
$\bA$ into $\bC^{\sigma_2}$. Then $\bA \in \Age(\bC)$, so there exists
an embedding $e \colon \bA \to \bC$. 
Let $\bar c \ceq (e(1),\dots,e(k))$. 
By the homogeneity of $\bC^{\sigma_i}$
there exists 
$\alpha_i \in \Aut(\bC^{\sigma_i})$ 
such that $\alpha_1(\bar c) = \bar a$
and $\alpha_2(\bar c) = \bar b$,
showing that $\bar c \in \Aut(\bB^{\tau_1}) \bar a \cap \Aut(\bB^{\tau_2}) \bar b$.   
Finally, $(\ref{eq:splits})$ follows directly from the homogeneity of $\bC$. 
\end{proof}

Note that properties $(\ref{eq:free})$ and $(\ref{eq:splits})$ for $\bB$, $\bB^{\tau_1}$, $\bB^{\tau_2}$
are equivalent to items \eqref{item:satisfiableTogether} and \eqref{item:canSeperateFormula} in Section~\ref{sect:qcsps} for $T=\Th(\bB)$, $T_1=\Th(\bB^{\tau_1})$, $T_2 = \Th(\bB^{\tau_2})$, respectively, because $\bB$ is $\omega$-categorical.
Lemma~\ref{lem:combine-2} motivates the following definition. 

\begin{definition}[Generic Combination]
Let $\bB_1$ and $\bB_2$
 be countably infinite $\omega$-categorical structures with disjoint relational signatures $\tau_1$ and  $\tau_2$, and let $\bB$ be a model of 
$\Th(\bB_1) \cup \Th(\bB_2)$.
If $\bB$ satisfies item $(\ref{eq:free})$
then we say that $\bB$ is a \emph{free combination} of $\bB_1$ and $\bB_2$.
If $\bB$ satisfies both item $(\ref{eq:free})$ and item $(\ref{eq:splits})$ then 
we say that $\bB$ is a \emph{generic combination} (or \emph{random combination}; see~\cite{AckermanFreerPatel}) 
of $\bB_1$ and $\bB_2$. 
\end{definition} 

 In later proofs we will need means to prove satisfiability of conjunctions of formulas in the presence of parameters. This is facilitated by the following lemma. For tuples $\bar{c}$ we will use $\Aut_{\bar{c}}(\bB)$ for $\set{\alpha\in \Aut(\bB) \mid \alpha(\bar{c}) = \bar{c}}$.

\begin{lemma}\label{lem:genericPropertiesWithParameters}
  Let $\bB$ be the generic combination of the structures $\bB_1$ and $\bB_2$ with relational signatures $\tau_1, \tau_2$ respectively. Let $B$ be the domain of $\bB$ and $\bar{a}, \bar{b}\in B^{n}$ injective tuples and $\bar{c} \in B^m$ such that all entries of $\bar{c}$ are distinct from all entries of $\bar{a}$ and $\bar{b}$. Then
  \begin{align}
    &\Aut_{\bar{c}}(\bB^{\tau_1})\bar{a} \cap \Aut_{\bar{c}}(\bB^{\tau_2}) \bar{b} \neq \emptyset \qquad \text{and} \label{eq:freeWithParams}\\
    &\Aut_{\bar{c}}(\bB^{\tau_1})\bar{a} \cap \Aut_{\bar{c}}(\bB^{\tau_2}) \bar{a} = \Aut_{\bar{c}}(\bB)\bar{a}.\label{eq:splitsWithParams}
  \end{align}
\end{lemma}

\begin{proof}
  There exists $\bar{d} \in \Aut(\bB^{\tau_1})(\bar{c},\bar{a}) \cap \Aut(\bB^{\tau_2})(\bar{c}, \bar{b})$ by Property~(\ref{eq:free}) of generic combination. Notice that we did not require $\bar{c}$ to be injective. We can nevertheless apply Property~(\ref{eq:free}) by removing duplicate entries in $\bar{c}$, applying the property to the injective version and reintroducing the duplicates afterwards. In particular, $(d_1, \dotsc, d_m) \in \Aut(\bB^{\tau_1})\bar{c} \cap \Aut(\bB^{\tau_2})\bar{c} = \Aut(\bB)\bar{c}$ by Property~(\ref{eq:splits}) of generic combination. Hence, there exists $\gamma\in \Aut(\bB)$ such that $\gamma((d_1, \dotsc, d_m)) = \bar{c}$. Furthermore, there exist $\alpha\in \Aut(\bB^{\tau_1}), \beta\in \Aut(\bB^{\tau_2})$ such that  $\alpha((\bar{c},\bar{a})) = \bar{d} = \beta((\bar{c},\bar{b}))$. Hence, $\gamma\circ \alpha \in \Aut_{\bar{c}}(\bB^{\tau_1})$ and $\gamma\circ \beta \in \Aut_{\bar{c}}(\bB^{\tau_2})$ and $(\gamma \circ \alpha)\bar{a} = (\gamma\circ \beta)\bar{b} \in \Aut_{\bar{c}}(\bB^{\tau_1})\bar{a} \cap \Aut_{\bar{c}}(\bB^{\tau_2}) \bar{b}$.
  
    For equation~\eqref{eq:splitsWithParams}, $\supseteq$ is trivial and $\subseteq$ is implied by Property~(\ref{eq:splits}) of generic combination as follows.
    Let $\bar{d}\in \Aut_{\bar{c}}(\bB^{\tau_1})\bar{a} \cap  \Aut_{\bar{c}}(\bB^{\tau_2})\bar{a}$. Then
    $(\bar{c},\bar{d}) \in \Aut(\bB^{\tau_1})(\bar{c},\bar{a}) \cap  \Aut(\bB^{\tau_2})(\bar{c},\bar{a})$ and therefore $(\bar{c},\bar{d})\in \Aut(\bB)(\bar{c},\bar{a})$. Hence, there exists $\gamma \in \Aut(\bB)$ such that $\gamma(\bar{c},\bar{a}) = (\bar{c},\bar{d})$, i.e., $\bar{d}\in \Aut_{\bar{c}}(\bB)\bar{a}$.
\end{proof}

Lemma~\ref{lem:genericPropertiesWithParameters}  can be used to show the uniqueness of generic combinations.

\begin{lemma}\label{lem:unique}
Let $\bB_1$ and $\bB_2$ be countable $\omega$-categorical relational structures. 
Then up to isomorphism, there is at most one generic combination of $\bB_1$ and $\bB_2$. 
\end{lemma}
\begin{proof}
  Let $\bB$ and $\bB'$ be two generic combinations of $\bB_1$ and $\bB_2$. We prove that $\bB$ and $\bB'$ are isomorphic by a back-and-forth argument. For this, we inductively extend a bijection $\alpha$ between finite subsets of $\bB$ and $\bB'$ such that $\alpha$ preserves all first-order $\tau_1$-formulas and $\tau_2$-formulas.
The empty map $\alpha$ trivially preserves all first-order formulas.
  Both structures satisfy the same first-order $\tau_1$-formulas  without free variables and $\tau_2$-formulas without free variables, since they are models of 
$\Th(\bB_1) \cup \Th(\bB_2)$.  
Now suppose that we have already constructed $\alpha$ for the finite substructure of $\bB$ induced by the elements of the tuple $\bar b \in B^n$ and want to extend $\alpha$ to another element 
$b_{n+1}$ of $\bB$. 
By the $\omega$-categoricity
of $T_i$, for $i=1$ and $i=2$, there exists an element $c_i$ of $\bB'$ such that 
$(\bar b,b_{n+1})$ satisfies the same formulas in $\bB^{\tau_i}$ as
$(\alpha(\bar b),c_i)$ satisfies in $(\bB')^{\tau_i}$. 
By Property~\eqref{eq:freeWithParams} of Lemma~\ref{lem:genericPropertiesWithParameters} there exists $c$ such that $(\alpha(\bar{b}),c)$ is in the same $(\bB')^{\tau_i}$ orbit as $(\alpha(\bar{b}),c_i)$ for $i=1$ and $i=2$. We define $\alpha(b_{n+1}) \ceq c$. Then, $(\alpha(\bar{b}),b_{n+1})$ satisfies the same formulas in $\bB'$ as $(\bar{b}, b_{n+1})$ satisfies in $\bB$ by Condition~\eqref{eq:splits} for generic combinations.  

Extending 
$\alpha^{-1}$ to another element of $\bB'$ is symmetric. This concludes the back-and-forth construction of an isomorphism between $\bB$ and $\bB'$. 
\end{proof}

Due to Lemma~\ref{lem:unique} we can now define $\bB_1 \ast \bB_2$ as \emph{the} generic combination of $\bB_1$ and $\bB_2$ for structures $\bB_1, \bB_2$ where a generic combination exists.

\ignore{
\begin{remark}
  This lemma implies that if a $\tau_1$-formula is satisfied by $t_1$ and a $\tau_2$-formula by $t_2$ we can find a solution for the conjunction of the formulas as long as the identifications between coordinates (and parameters) in $t_1, t_2$ match. In that case we can enforce injectivity on a syntactical level before applying the lemma via suitable substitutions and revert the substitutions afterwards. As we will guarantee injectivity before applying the lemma in this article, we will not elaborate this further.
\end{remark}
}

We now prove Proposition~\ref{prop:r1} that we already stated in the introduction, and which states that two countably infinite $\omega$-categorical structures with disjoint relational signatures have a generic combination if and only if both have no algebraicity, or at least one of the structures has an automorphism group which is $n$-transitive for all $n \in {\mathbb N}$. 

\begin{proof}[Proof of Proposition~\ref{prop:r1}]
If both $\bB_1$ and $\bB_2$ do not have algebraicity then 
the existence of an $\omega$-categorical generic combination follows from Lemma~\ref{lem:combine-2}.
If on the other hand one of them, say $\bB_1$, is $n$-transitive for all $n$ then $\bB_1$ is isomorphic to a first-order reduct of $(\N;=)$. The generic combination $\bB_1 \ast \bB_2$ can then easily be seen to be a first-order expansion of $\bB_2$.

We prove the converse direction by contradiction. Let $\bB$ be the generic combination of the $\tau_1$-structure $\bB_1$ and the $\tau_2$-structure $\bB_2$. 
Recall that $\bB^{\tau_i}$ is isomorphic to $\bB_i$, for $i \in \{1,2\}$.
By symmetry between $\bB_1$ and $\bB_2$, we will assume towards a contradiction that $\bB^{\tau_1}$ has algebraicity and $\Aut(\bB^{\tau_2})$ is not $n$-transitive for some $n\in \mathbb{N}$.
Choose $n$ to be smallest possible, so that $\Aut(\bB^{\tau_2})$ is not $n$-transitive. Therefore there exist tuples $(b_0, \dotsc, b_{n-1})$ and  $(c_0,\dotsc, c_{n-1})$ in $B^n$, each with pairwise distinct entries, that are in different orbits with respect to $\Aut(\bB^{\tau_2})$. By the minimality of $n$, there exists $\alpha\in \Aut(\bB^{\tau_2})$ such that  $\alpha(b_1,\dots,b_{n-1}) = (c_1,\dots,c_{n-1})$.
The algebraicity of $\bB^{\tau_1}$ implies that there exists a first-order $\tau_1$-formula $\phi(x_0,x_1,\dots,x_m)$ 
and pairwise distinct elements $a_1,\dots,a_m$ of $B$ such that $\phi(x,a_1,\dots,a_m)$ holds for precisely one element $a_0$ other than $a_1,\dots,a_m$. 
By adding unused extra variables to $\phi$ we can assume that $m \geq n-1$.
Choose elements $b_n,\dotsc, b_m$ of $B$ such that the entries of $(b_0, \dotsc, b_{n-1}, b_n,\dotsc, b_m)$ are pairwise distinct and define $c_i \ceq \alpha(b_i)$ for $i\in \set{n,\dotsc, m}$.
Since $\bB$ is a free combination, there exist tuples $(b_0', \dotsc, b_m'),\, (c_0', \dotsc, c_m')$ and $\beta_1, \gamma_1 \in \Aut(\bB^{\tau_1})$ and $\beta_2, \gamma_2 \in \Aut(\bB^{\tau_2})$ such that
\begin{align*}
  \beta_2(b_0, \dotsc, b_m) = (b_0', \dotsc, b_m'), & \qquad \beta_1 (b_0', \dotsc, b_m') = (a_0, \dotsc, a_m), \\
  \gamma_2(c_0, \dotsc, c_m) = (c_0', \dotsc, c_m'), & \qquad  \gamma_1 (c_0', \dotsc, c_m') = (a_0, \dotsc, a_m).
\end{align*}
Because $\gamma_1^{-1}\circ \beta_1 \in \Aut(\bB^{\tau_1})$ and $\gamma_2 \circ \alpha \circ \beta_2^{-1} \in \Aut(\bB^{\tau_2})$ both map $(b_1', \dotsc, b_m')$ to $(c_1', \dotsc, c_m')$, and due to Condition~\eqref{eq:splits} for generic combinations, there exists $\mu \in \Aut(\bB)$ such that $\mu (b_1', \dotsc, b_m') = (c_1', \dotsc, c_m')$. Since any operation in $\Aut(\bB^{\tau_1})$ preserves $\phi$, we have $\gamma_1 \circ \mu \circ \beta_1^{-1} (a_0,\dots,a_m) = (a_0,\dots,a_m)$. Therefore $\mu$ must map $b_0'$ to  $c_0'$. Hence, $\gamma_2^{-1} \circ \mu \circ \beta_2 \in \Aut(\bB^{\tau_2})$ maps $(b_0, \dotsc, b_{n-1})$ to $(c_0, \dotsc, c_{n-1})$, contradicting our assumption that they lie in different orbits with respect to $\Aut(\bB^{\tau_2})$.
\end{proof}

\section{On the Necessity of the Nelson-Oppen Conditions}
\label{sect:bin-inj}
In this section we prove that for theories with \hyperlink{def:J}{Property~$J$} the conditions of
Nelson and Oppen are not only sufficient, but also necessary for the polynomial-time tractability of generic combinations (unless P = NP). In particular, we prove Theorem~\ref{thm:r2} from the introduction. 
In order to avoid an excess of superscripts we will drop the bar above tuples in this section. We need the following characterisation of
convexity of $\omega$-categorical theories.
\begin{thmC}[{\cite[Lemma 6.1.3]{Bodirsky-HDR-v8}}]\label{thm:hdr}

Let $\bB$ be a countably infinite $\omega$-categorical relational structure and let $T$ be its first-order theory. 
Then the following are equivalent. 
\begin{itemize}
\item $T$ is convex; 
\item $\bB$ has a binary injective polymorphism.
\end{itemize}
Moreover, if $\bB$ contains the relation $\neq$,
these conditions are also equivalent to the following.
\begin{itemize}
\item If $S$ is a finite set of atomic $\tau$-formulas
 such that 
$S \cup T \cup \{x_1 \neq y_1\}$ is satisfiable and $S \cup T \cup \{x_2 \neq y_2\}$
is satisfiable, then 
 $T \cup S \cup \{x_1 \neq y_1,x_2 \neq y_2\}$ is satisfiable, too. 
\end{itemize}
\end{thmC}
\noindent The construction of the binary injective polymorphism, which is at the core of this theorem, first finds an injective homomorphism from a finite substructure of $\bB^2$ to $\bB$ and then uses compactness of first-order logic to expand its domain to $\bB^2$. 

 An operation $f \colon B^k \to B$ is called \emph{essentially unary} if there exists an $i \leq k$ and a function $g \colon B \to B$ such that $f(x_1,\dots,x_k) = g(x_i)$ for all $x_1,\dots,x_k \in B$. The operation $f$ is called \emph{essential} if it is not essentially unary.
Hence, if a function $f\colon B^2 \rightarrow B$ is essential then there exist $\mybar{p},\mybar{q},\mybar{r},\mybar{s}\in B^2$ such that $p_2 = q_2$ and $r_1 = s_1$ and
$f(p_1, p_2) \neq  f(q_1, q_2)$ and $f(r_1, r_2) \neq  f(s_1, s_2)$.
The tuples $(\mybar{p},\mybar{q})$ and $(\mybar{r},\mybar{s})$ are called \emph{witnesses} for the essentiality of $f$.


The following fact is well-known. We give a short proof to illustrate how it follows from known results.

\begin{proposition}\label{prop:ecsp-hard}
Let $\bB$ be an infinite $\omega$-categorical structure with finite relational signature containing the relation $\neq$ and such that all polymorphisms of $\bB$ are essentially unary. Then $\Csp(\bB)$ is NP-hard. 
\end{proposition}
\begin{proof}
  Let $E$ be the set of all relations (on the domain of $\bB$) that have a first-order definition over the empty signature.
  If all polymorphisms of $\bB$ are essentially unary and preserve $\neq$, they preserve all relations in $E$. Therefore, all relations in $E$ are pp-definable in $\bB$ (see~\cite{BodirskyNesetrilJLC}, Theorem 4). In this case, there is a finite signature reduct of $\bB$ which has an NP-hard CSP (see~\cite{ecsps}, Theorem~1).
\end{proof}

Hence, it suffices to show that the existence of an essential polymorphism of the generic combination of two countably infinite $\omega$-categorical structures $\bB_1$ and $\bB_2$ implies the existence of a binary injective polymorphism. The key technical result is the following proposition.

  \begin{proposition}\label{prop:key}
   Let $\bB$ be the generic combination of two $\omega$-categorical relational structures $\bB_1$ and $\bB_2$ such that the following holds:
   \begin{itemize}
   \item $\bB$ has a binary essential polymorphism,
   \item both $\bB_1$ and $\bB_2$ are without algebraicity,
   \item $\bB_1$ can prevent crosses,
   \item $\bB_2$ is 2-set-transitive and $\neq$ is pp-definable in $\bB_2$.
   \end{itemize}
Then $\bB_2$  has a binary injective polymorphism.
\end{proposition}

To make the technical part of the proof of Proposition~\ref{prop:key} more accessible and clarify which condition is needed where, we will break it into two lemmata.

\begin{lemma} \label{lem:noCrossPatternPossible}
  Let $\bB$ be the generic combination of two $\omega$-categorical relational structures $\bB_1$ and $\bB_2$.  Let $\bB_1$ be a structure which can prevent crosses and without algebraicity and let $\varphi(x_1, x_2, x_1', y_1, y_2, y_2';\mybar{a})$ be a first-order formula over $\bB_2$ with parameter tuple $\mybar{a}$ such that there exists a solution $s$ for $\varphi$ which is injective and the image of $s$ does not contain any value from $\mybar{a}$. Then, for any binary polymorphism $f\in \Pol(\bB_1)$ there exist $x_1, x_2, x_1', y_1, y_2, y_2'$ such that
  \begin{equation}
  \varphi(x_1, x_2, x_1', y_1, y_2, y_2',\mybar{a}) \AND \big( f(x_1,x_2) \neq f(x_1', x_2) \OR  f(y_1, y_2) \neq f(y_1, y_2') \big). \label{eq:crossEqualities}
  \end{equation}
\end{lemma}

\begin{proof}
  Let $\theta$ be a cross prevention formula of $\bB_1$. Define
  \begin{displaymath}
    \rho(x_1, x_2, x_1', y_1,y_2, y_2') \cequiv \theta(x_1,x_1',y_1,y_1) \AND \theta(x_2,x_2,y_2,y_2').
  \end{displaymath}
 Because $\bB_1$ is without algebraicity and there are injective solutions for $\theta(x_1,x_1',y_1,y_1)$ and $\theta(x_2,x_2,y_2,y_2')$ we can use Neumann's lemma (see Corollary 4.2.2 in \cite{HodgesLong}) to find a solution $\mybar{t}_1$ for $\rho$ that is injective and avoids all entries in $\mybar{a}$ in its image. Hence, there are injective solutions $\mybar{t}_1, \mybar{t}_2$ for $\rho, \varphi$ respectively that avoid all values in $\mybar{a}$ in their image and therefore Lemma~\ref{lem:genericPropertiesWithParameters} is applicable and there exists a solution $\mybar{t}$ for $\rho \AND \varphi$.\\
  We now look at the images of $f$:
  \begin{align*}
    f(t(x_1), t(x_2)) &= x\\
    f(t(x_1'),t(x_2)) &= x'\\
    f(t(y_1), t(y_2)) &= y\\
    f(t(y_1), t(y_2')) &= y'
  \end{align*}

As $f$ is a polymorphism of $\bB_1$ the tuple $(x,x',y,y')$ must satisfy $\theta$. Hence, either $x\neq x'$ or $y\neq y'$.
\end{proof}

\begin{lemma}\label{lem:witnessesInAllOrbits}
  Let $\bB$ be the generic combination of two $\omega$-categorical relational structures $\bB_1$ and $\bB_2$,  both of which are without algebraicity.  Furthermore, assume that $\bB_1$ can prevent crosses and that $\bB_2$ is 2-set-transitive. Let $B$ be the domain of $\bB$ and choose $\mybar{a},\mybar{b},\mybar{c},\mybar{d}\in B^2$ such that $a_1 \neq b_1,\,  a_2 = b_2$  and $c_1 = d_1, c_2 \neq d_2$.
  
 Then, for any binary essential polymorphism $f\in \Pol(\bB_1)$ there exist $\mybar{\tilde{a}}, \mybar{\tilde{b}}, \mybar{\tilde{c}}, \mybar{\tilde{d}}\in B^2$ such that $f(\tilde{a}_1,\tilde{a}_2) \neq f(\tilde{b}_1, \tilde{b}_2)$ and $f(\tilde{c}_1,\tilde{c}_2) \neq f(\tilde{d}_1, \tilde{d}_2)$ and $(\tilde{a}_i, \tilde{b}_i, \tilde{c}_i, \tilde{d}_i)\in \Aut(\bB_2)(a_i,b_i,c_i, d_i)$ for $i=1$ and $i=2$.
\end{lemma}

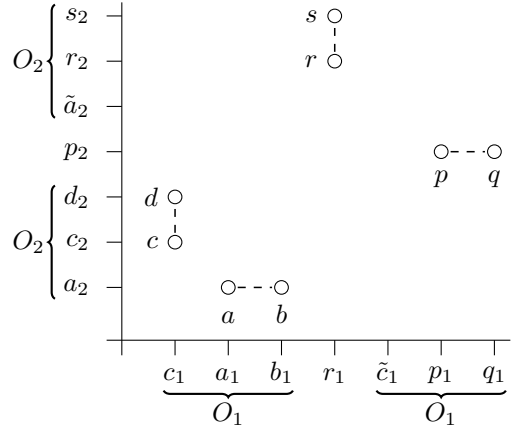
\begin{figure}
  \begin{center}
  \begin{tikzpicture}
    \def\xbetweenspace{0.7};
    \def\ybetweenspace{0.6};
    \def\belowspace{-0.5};
    \def\belowspaceinside{-0.6};
    \def\leftspace{0.1};
    \def\xoffset{0.7};
    \def\yoffset{0.7};
    \def\bracexextra{0.15};
    \def\braceyextra{0.15};
    
    \myhposition{\xoffset+0*\xbetweenspace}{0}{1}{node[align=left, above=\belowspace] (c1) {$c_1$}};
    \myhposition{\xoffset+1*\xbetweenspace}{0}{2}{node[align=left, above=\belowspace] (a1) {$a_1$}};
    \myhposition{\xoffset+2*\xbetweenspace}{0}{3}{node[align=left, above=\belowspace] (b1) {$b_1$}};
    \myhposition{\xoffset+3*\xbetweenspace}{0}{3b}{node[align=left, above=\belowspace] (b1) {$r_1$}};
    \draw[thick, decoration={brace,mirror,raise=1.9em},decorate] (\xoffset+0*\xbetweenspace-\bracexextra,0) -- (\xoffset+2*\xbetweenspace+\bracexextra,0) node[pos=0.5,anchor=north,yshift=-2.0em] {$O_1$};
    \myhposition{\xoffset+4*\xbetweenspace}{0}{4}{node[align=left, above=\belowspace] (ct) {$\tilde{c}_1$}};
    \myhposition{\xoffset+5*\xbetweenspace}{0}{5}{node[align=left, above=\belowspace] (p1) {$p_1$}};
    \myhposition{\xoffset+6*\xbetweenspace}{0}{6}{node[align=left, above=\belowspace] (q1) {$q_1$}};
    \draw[thick, decoration={brace,mirror,raise=1.9em},decorate] (\xoffset+4*\xbetweenspace-\bracexextra,0) -- (\xoffset+6*\xbetweenspace+\bracexextra,0) node[pos=0.5,anchor=north,yshift=-2.0em] {$O_1$};

    \myvposition{0}{\yoffset+0*\ybetweenspace}{7}{node[align=right, left=\leftspace] {$a_2$}};
    \myvposition{0}{\yoffset+1*\ybetweenspace}{8}{node[align=right, left=\leftspace] {$c_2$}};
    \myvposition{0}{\yoffset+2*\ybetweenspace}{9}{node[align=right, left=\leftspace] {$d_2$}};
    \myvposition{0}{\yoffset+3*\ybetweenspace}{9b}{node[align=right, left=\leftspace] {$p_2$}};
    \draw[thick, decoration={brace,raise=2.5em},decorate] (0,\yoffset+0*\ybetweenspace-\braceyextra) -- (0,\yoffset+2*\ybetweenspace+\braceyextra) node[pos=0.5,anchor=east,xshift=-2.5em] {$O_2$};
    \myvposition{0}{\yoffset+4*\ybetweenspace}{10}{node[align=right, left=\leftspace] {$\tilde{a}_2$}};
    \myvposition{0}{\yoffset+5*\ybetweenspace}{11}{node[align=right, left=\leftspace] {$r_2$}};
    \myvposition{0}{\yoffset+6*\ybetweenspace}{12}{node[align=right, left=\leftspace] {$s_2$}};
    \draw[thick, decoration={brace,raise=2.5em},decorate] (0,\yoffset+4*\ybetweenspace-\braceyextra) -- (0,\yoffset+6*\ybetweenspace+\braceyextra) node[pos=0.5,anchor=east,xshift=-2.5em] {$O_2$};

    \mynode{\xoffset+1*\xbetweenspace}{\yoffset+0*\ybetweenspace}{a}{node[align=right, above=\belowspaceinside] {$a$}};
    \mynode{\xoffset+2*\xbetweenspace}{\yoffset+0*\ybetweenspace}{b}{node[align=right, above=\belowspaceinside] {$b$}};
    \mynode{\xoffset+5*\xbetweenspace}{\yoffset+3*\ybetweenspace}{p}{node[align=right, above=\belowspaceinside] {$p$}};
    \mynode{\xoffset+6*\xbetweenspace}{\yoffset+3*\ybetweenspace}{q}{node[align=right, above=\belowspaceinside] {$q$}};
    \mynode{\xoffset+0*\xbetweenspace}{\yoffset+1*\ybetweenspace}{c}{node[align=right, left=\leftspace] {$c$}};
    \mynode{\xoffset+0*\xbetweenspace}{\yoffset+2*\ybetweenspace}{d}{node[align=right, left=\leftspace] {$d$}};
    \mynode{\xoffset+3*\xbetweenspace}{\yoffset+5*\ybetweenspace}{r}{node[align=right, left=\leftspace] {$r$}};
    \mynode{\xoffset+3*\xbetweenspace}{\yoffset+6*\ybetweenspace}{s}{node[align=right, left=\leftspace] {$s$}};

    \draw[different] (a) -- (b);
    \draw[different] (c) -- (d);
    \draw[different] (p) -- (q);
    \draw[different] (r) -- (s);

    \systemAxis{0.0}{0.0}{\xoffset+6.3*\xbetweenspace}{\yoffset+6.3*\ybetweenspace}; 
    
  \end{tikzpicture}
\end{center}
\caption{An illustration of Case 1 in the proof of Lemma~\protect\ref{lem:witnessesInAllOrbits}. Dashed edges indicate (potential) disequalities between function values.}
\label{fig:Case1}
\end{figure}

\begin{proof}
Case 1: Both $(c_1, a_1, b_1)$ and $(a_2, c_2,d_2)$ are injective tuples (see Figure~\ref{fig:Case1}).

  Define $O_1 \ceq \Aut(\bB_2)(c_1, a_1,b_1)$ and $O_2 \ceq \Aut(\bB_2)(a_2, c_2,d_2)$. Notice that $O_1$ and $O_2$ can be defined by first-order formulas because $\bB_2$ is $\omega$-categorical. We start by using one of the witnesses of essentiality of $f$ that is guaranteed to exist by definition, i.e., $\mybar{p},\mybar{q}\in B^2$ such that  $f(p_1,p_2) \neq f(q_1, q_2)$ with $p_2 = q_2$. Due to 2-set-transitivity, we can assume that there exists $\alpha \in \Aut(\bB_2)$ such that $\alpha(p_1) = a_1$ and $\alpha(q_1) = b_1$ (otherwise, swap $p$ and $q$).
  Now look at all $\tilde{c}_1\in B$ such that $( \tilde{c}_1, p_1, q_1)\in O_1$. 
  If for $\mybar{\tilde{a}} \ceq \mybar{p},\, \mybar{\tilde{b}} \ceq \mybar{q}$ there are $\tilde{c}_2,\,\tilde{d}_2$ such that $(\mybar{\tilde{a}}, \mybar{\tilde{b}},\mybar{\tilde{c}},\,\mybar{\tilde{d}})$ is a witness for the statement of the lemma, we are done. Otherwise, for all $\tilde{c}_1, \tilde{c}_2,  \tilde{d}_2 \in B$ the following holds:
  \begin{displaymath}
    (\tilde{c}_1, p_1, q_1)\in O_1 \AND (p_2, \tilde{c}_2,  \tilde{d}_2)\in O_2 \Rightarrow f( \tilde{c}_1,  \tilde{c}_2) = f( \tilde{c}_1, \tilde{d}_2)
  \end{displaymath}
  In this case we consider the second witness of essentiality of $f$ guaranteed to exist by definition, i.e., $\mybar{r},\mybar{s}\in B^2$ with  $f(r_1,r_2) \neq f(s_1, s_2)$ and $r_1 = s_1$.
 Again due to 2-set-transitivity, there exists $\beta \in \Aut(\bB_2)$ such that $\beta(r_2) = c_2$ and $\beta(s_2) = d_2$ (otherwise swap $r$ and $s$). Assume now, that for all $\tilde{a}_1, \tilde{b}_1,  \tilde{a}_2 \in B$ the following holds:
  \begin{equation}\label{eq:noGoodABTilde}
   (r_1, \tilde{a}_1, \tilde{b}_1)\in O_1 \AND  ( \tilde{a}_2, r_2, s_2)\in O_2 \Rightarrow f(\tilde{a}_1, \tilde{a}_2) = f(\tilde{b}_1, \tilde{a}_2).
  \end{equation}
We now argue why we can apply Lemma~\ref{lem:noCrossPatternPossible} to
  \begin{align*}
   \phi(x_1, x_2, x_1', y_1,y_2, y_2'; p_1, p_2, q_1, r_1, r_2, s_2) \cequiv &\, (y_1, p_1,q_1)\in O_1 \AND (p_2,y_2,y'_2)\in O_2  \AND \\
                       & \, (r_1, x_1,x'_1)\in O_1 \AND (x_2,r_2, s_2)\in O_2,
  \end{align*}
  where $\mybar{p},\mybar{q},\mybar{r},\mybar{s}$ are parameters.
  Define the set $T \ceq \set{y_1 \mid (y_1, p_1, q_1)\in O_1}$.  $T$ is non-empty because there is an automorphism $\alpha$ such that $\alpha(p_1) = a_1$ and $\alpha(q_1) = b_1$ and hence $\alpha^{-1}(c_1) \in T$. Furthermore, $a_1 \neq c_1 \neq b_1$ and hence $p_1 \neq \alpha^{-1}(c_1) \neq q_1$. Therefore, $T$ is infinite because $\bB_2$ has no algebraicity. 
  In a similar way, when considering all solutions for any other conjunct of $\phi$, each variable can take infinitely many values. For the two conjuncts with two free variables Neumann's lemma guarantees the existence of infinitely many tuples with pairwise distinct entries in the orbit. Therefore, $\phi$ is a first-order formula with parameters in $\bB_2$ which has an injective solution without any of the parameter values in its image.
  
An application of Lemma~\ref{lem:noCrossPatternPossible} yields that Implication~(\ref{eq:noGoodABTilde}) must be false. Hence, there are  $\tilde{a}_1, \tilde{b}_1,  \tilde{a}_2 \in B$ such that  $(r_1, \tilde{a}_1, \tilde{b}_1)\in O_1$ and $(\tilde{a}_2, r_2, s_2)\in O_2$ and $f(\tilde{a}_1, \tilde{a}_2) \neq f(\tilde{b}_1, \tilde{a}_2)$. Together with $\mybar{\tilde{c}} \ceq \mybar{r}$ and $\mybar{\tilde{d}} \ceq \mybar{s}$, the points $\mybar{\tilde{a}}, \mybar{\tilde{b}}, \mybar{\tilde{c}}, \mybar{\tilde{d}}$ satisfy all conditions from the statement.\\

  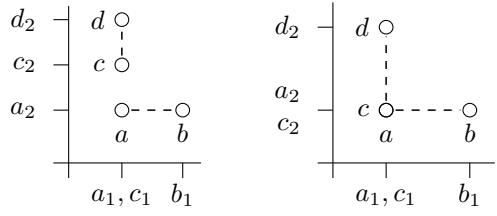
\begin{figure}
  \begin{center}
  \begin{tikzpicture}
    \def\xbetweenspace{0.8};
    \def\ybetweenspace{0.6};
    \def\belowspace{-0.5};
    \def\belowspaceinside{-0.55};
    \def\leftspace{0.1};
    \def\xoffset{0.7};
    \def\yoffset{0.7};
    
    \myhposition{\xoffset+0*\xbetweenspace}{0}{1}{node[align=left, above=\belowspace] {$a_1, c_1$}};
    \myhposition{\xoffset+1*\xbetweenspace}{0}{2}{node[align=left, above=\belowspace] {$b_1$}};
    \myvposition{0}{\yoffset+0*\ybetweenspace}{4}{node[align=right, left=\leftspace] {$a_2$}};
    \myvposition{0}{\yoffset+1*\ybetweenspace}{5}{node[align=right, left=\leftspace] {$c_2$}};
    \myvposition{0}{\yoffset+2*\ybetweenspace}{6}{node[align=right, left=\leftspace] {$d_2$}};
    \mynode{\xoffset+0*\xbetweenspace}{\yoffset+0*\ybetweenspace}{a}{node[align=right, above=\belowspaceinside] {$a$}};
    \mynode{\xoffset+1*\xbetweenspace}{\yoffset+0*\ybetweenspace}{b}{node[align=right, above=\belowspaceinside] {$b$}};
    \mynode{\xoffset+0*\xbetweenspace}{\yoffset+1*\ybetweenspace}{c}{node[align=right, left=\leftspace] {$c$}};
    \mynode{\xoffset+0*\xbetweenspace}{\yoffset+2*\ybetweenspace}{d}{node[align=right, left=\leftspace] {$d$}};

    \draw[different] (a) -- (b);
    \draw[different] (c) -- (d);

    \systemAxis{0.0}{0.0}{\xoffset+1.3*\xbetweenspace}{\yoffset+2.3*\ybetweenspace}; 
    
  \end{tikzpicture}  \qquad
  \begin{tikzpicture}
    \def\xbetweenspace{1.1};
    \def\ybetweenspace{1.1};
    \def\belowspace{-0.5};
    \def\belowspaceinside{-0.55};
    \def\leftspace{0.1};
    \def\xoffset{0.7};
    \def\yoffset{0.7};
    
    \myhposition{\xoffset+0*\xbetweenspace}{0}{1}{node[align=left, above=\belowspace] {$a_1,c_1$}};
    \myhposition{\xoffset+1*\xbetweenspace}{0}{2}{node[align=left, above=\belowspace] {$b_1$}};
    \myvposition{0}{\yoffset+0*\ybetweenspace}{4}{node[align=right, left=\leftspace] {$a_2$\\$c_2$}};
    \myvposition{0}{\yoffset+1*\ybetweenspace}{6}{node[align=right, left=\leftspace] {$d_2$}};
    \mynode{\xoffset+0*\xbetweenspace}{\yoffset+0*\ybetweenspace}{a}{node[align=right, above=\belowspaceinside] {$a$}};
    \mynode{\xoffset+1*\xbetweenspace}{\yoffset+0*\ybetweenspace}{b}{node[align=right, above=\belowspaceinside] {$b$}};
    \mynode{\xoffset+0*\xbetweenspace}{\yoffset+0*\ybetweenspace}{c}{node[align=right, left=\leftspace] {$c$}};
    \mynode{\xoffset+0*\xbetweenspace}{\yoffset+1*\ybetweenspace}{d}{node[align=right, left=\leftspace] {$d$}};

    \draw[different] (a) -- (b);
    \draw[different] (c) -- (d);

    \systemAxis{0.0}{0.0}{\xoffset+1.3*\xbetweenspace}{\yoffset+1.3*\ybetweenspace}; 
    
  \end{tikzpicture}  
\end{center}
\caption{An illustration of the subcases of Case 2 (up to permutation of variables) in the proof of Lemma~\protect\ref{lem:witnessesInAllOrbits}. Dashed edges indicate (potential) disequalities between function values.}
\label{fig:Case2}
\end{figure}
Case 2: $(c_1, a_1,b_1)$ or $(a_2, c_2, d_2)$  are not injective, i.e., $c_1 = a_1$ or $c_1 = b_1$ or $a_2 =c_2$ or $a_2=d_2$ (see Figure~\ref{fig:Case2}).

In this case, the witnesses can be constructed using Case 1 by starting with suitable injective tuples.
If $c_1 = a_1$ we consider a new tuple $(c'_1, a'_1, b'_1)$ which is injective such that the $\bB_2$ orbits of $(c_1, b_1)$, $(c'_1, a'_1)$ and $(c'_1, b'_1)$ are the same (for $c_1 = b_1$ we do the same but with the $\bB_2$ orbit of $(c_1,a_1)$ instead of $(c_1,b_1)$). This tuple exists due to no-algebraicity of $\bB_2$. Likewise, if  $a_2 = c_2$ we can find $(a'_2, c'_2, d'_2)$ injective such that the $\bB_2$ orbits of $(a_2, d_2)$, $(a'_2,c'_2)$ and $(a'_2, d'_2)$ are the same (for $a_2 = d_2$ we use the $\bB_2$ orbit of $(a_2,c_2)$ instead of $(a_2, d_2)$). If $(c_1, a_1,b_1)$ or $(a_2, c_2, d_2)$ was injective already, we keep this tuple but name them $a_i'$ etc. to simplify the presentation.
  
Now, we can apply Case 1 to $(c'_1, a'_1, b'_1)$ and $(a'_2, c'_2, d'_2)$ to find
$\mybar{\tilde{a}}', \mybar{\tilde{b}}', \mybar{\tilde{c}}', \mybar{\tilde{d}}'\in B^2$  as promised by this lemma.

   For $c_1=a_1$  both $(\tilde{c}'_1,\tilde{a}'_1, \tilde{c}'_1, \tilde{c}'_1)$ and $(\tilde{c}'_1,\tilde{b}'_1, \tilde{c}'_1, \tilde{c}'_1)$ are in the same orbit as $(a_1, b_1, c_1, d_1)$ because  $(\tilde{c}'_1, \tilde{a}'_1), (\tilde{c}'_1, \tilde{b}'_1)$ and $(c_1, b_1)$ have the same orbit and $a_1 = c_1 = d_1$. Likewise for $a_2 = c_2$ in the second coordinate, both $(\tilde{a}'_2,\tilde{a}'_2, \tilde{a}'_2,\tilde{d}'_2)$ and $(\tilde{a}'_2,\tilde{a}'_2, \tilde{a}'_2,\tilde{c}'_2)$ are in the same orbit as $(a_2, b_2, c_2, d_2)$.
     
   However, if $f(\tilde{a}'_1,\tilde{a}'_2) \neq f(\tilde{b}'_1, \tilde{a}'_2)$ then either $f(\tilde{a}'_1,\tilde{a}'_2) \neq f(\tilde{c}'_1, \tilde{a}'_2)$ or $f(\tilde{b}'_1,\tilde{a}'_2) \neq f(\tilde{c}'_1, \tilde{a}'_2)$. Likewise, either  $f(\tilde{c}'_1,\tilde{a}'_2) \neq f(\tilde{c}'_1, \tilde{c}'_2)$ or $f(\tilde{c}'_1,\tilde{a}'_2) \neq f(\tilde{c}'_1, \tilde{d}'_2)$.   
Therefore, there is a tuple $(\mybar{a}', \mybar{b}', \mybar{c}', \mybar{d}') = ((a_1', b_1', c_1', d_1'),(a_2',b_2', c_2', d_2'))$ in
  \begin{displaymath}
    \set{(\tilde{c}'_1,\tilde{a}'_1, \tilde{c}'_1, \tilde{c}'_1),\, (\tilde{c}'_1,\tilde{b}'_1, \tilde{c}'_1, \tilde{c}'_1)} \times \set{(\tilde{a}'_2,\tilde{a}'_2, \tilde{a}'_2,\tilde{c}'_2),\,(\tilde{a}'_2,\tilde{a}'_2, \tilde{a}'_2,\tilde{d}'_2)},
  \end{displaymath}
which satisfies the lemma.

  The procedure is analogous for other identifications.
 \end{proof}

 Notice that Lemma~\ref{lem:witnessesInAllOrbits} also holds for polymorphisms of $\bB$ because $\Pol(\bB) \subseteq \Pol(\bB_1)$. We are now ready to prove Proposition~\ref{prop:key}.

\begin{proof}[Proof of Proposition~\ref{prop:key}]
Let $\phi$ be a (finite) conjunction of atomic formulas over $\bB_2$ such that $\phi\AND x\neq y$ and $\phi\AND u\neq v$ are satisfiable and $s_1,\, s_2$ are satisfying assignments respectively. As $\bB_2$ is $\omega$-categorical and $\neq$ is pp-definable in $\bB_2$, Theorem~\ref{thm:hdr} is applicable. Hence, it is sufficient to prove satisfiability of $\phi\AND x\neq y \AND u\neq v$ in $\bB_2$.
  Let $f\in \Pol(\bB)$ be a binary essential polymorphism. If $s_1(u)\neq s_1(v)$ or $s_2(x)\neq s_2(y)$ we have found a satisfying assignment for $\phi\AND x\neq y \AND u\neq v$, which is what we wanted. Otherwise, due to Lemma~\ref{lem:witnessesInAllOrbits} above, there are witnesses $\mybar{\tilde{x}},\, \mybar{\tilde{y}}, \mybar{\tilde{u}}, \mybar{\tilde{v}}\in B^2$ of essentiality of $f$ such that $(\tilde{x}_1, \tilde{y}_1, \tilde{u}_1)$ is in the same $\bB_2$ orbit as $s_1(x, y, u)$ and $(\tilde{x}_2,\tilde{u}_2, \tilde{v}_2)$ is in the same $\bB_2$ orbit as $s_2(x, u, v)$. Hence, there exist automorphisms $\alpha_1, \alpha_2\in \Aut(\bB_2)$, such that
  \begin{align*}
    f(\alpha_1(s_1(x)), \alpha_2(s_2(x))) &= f(\tilde{x}_1, \tilde{x}_2) \\
                                          &\neq  f(\tilde{y}_1, \tilde{x}_2)\\
                                          &=  f(\alpha_1(s_1(y)), \alpha_2(s_2(x)))\\
                                          &= f(\alpha_1(s_1(y)), \alpha_2(s_2(y))) \quad \text{and}\\
    f(\alpha_1(s_1(u)), \alpha_2(s_2(u)))   &= f(\tilde{u}_1, \tilde{u}_2)\\
                                          &\neq  f(\tilde{u}_1, \tilde{v}_2)\\
                                          &=  f(\alpha_1(s_1(u)), \alpha_2(s_2(v)))\\
                                          &= f(\alpha_1( s_1(v)), \alpha_2(s_2(v))).
  \end{align*}
Therefore, $z\mapsto f((\alpha_1\circ s_1)(z), (\alpha_2\circ s_2)(z))$ for any variable $z$ that occurs in  $\phi\AND x\neq y \AND u\neq v$ is a solution for $\phi\AND x\neq y \AND u\neq v$.
\end{proof}

To apply Proposition~\ref{prop:key}, we need to prove the existence of binary essential polymorphisms of generic combinations $\bB$. 
For this, we use an idea that first appeared in~\cite{tcsps-journal} and was later generalized
in~\cite{Bodirsky-HDR-v8}, based on the following concept.
A permutation group $G$ on a set $B$ has the \emph{orbital extension property (OEP)} if there is an orbital $O$ such that for all $b_1,b_2 \in B$ there is an element $c \in B$ where $(b_1,c) \in O$ and $(b_2,c) \in O$. Notice that such $c$ must always be distinct from $b_1$ and $b_2$. The relevance of the OEP comes from the following lemma.

\begin{lemma}[K\'ara's Lemma; see~{\cite[Lemma 5.3.10]{Bodirsky-HDR-v8}}]
\label{lem:kara}
Let $\bB$ be a structure with an essential polymorphism and an automorphism group with the OEP. Then $\bB$ must have a binary essential polymorphism. 
\end{lemma}

To apply this lemma to the generic combination $\bB$ of $\bB_1$ and $\bB_2$, we have to verify that $\Aut(\bB)$ has the OEP. 

\begin{lemma}\label{lem:orb-ext}
Any 2-set-transitive permutation group on a set with at least 4 elements has the OEP. 
\end{lemma}
\begin{proof}
  Let $G$ be a 2-set-transitive permutation group on a set $B$. 
If $G$ is even 2-transitive then the statement is obvious. Otherwise, 
observe that there are exactly two orbitals $O$ and $P$ such that 
$x \neq y$ if and only if $(x,y) \in O$ and $(y,x) \in P$. Let $b_1,b_2 \in B$ be distinct. 
Consider the tournament induced on $B$ by the edge relation $O$ and let $(\set{x,y,u,v}; O)$ be any 4-element substructure of $(B; O)$. There exists a node in $\set{x,y,u,v}$, say $x$, with two incoming edges, say $(y,x), (u,x)\in O$. By 2-set-transitivity there exists $\alpha \in G$ such that $\alpha(\set{b_1, b_2}) = \set{y,u}$. Choose $c \ceq \alpha^{-1}(x)$ and we get $(b_1, c), (b_2, c) \in O$.
\end{proof}

\begin{lemma}\label{lem:free-orb-ext}
Let $\bB$ be a generic combination of two $\omega$-categorical relational structures $\bB_1$ and $\bB_2$ with the OEP. Then
$\bB$ has the OEP.
\end{lemma}
\begin{proof}
For $i=1$ and $i=2$, 
let $O_i$ be the orbital that witnesses the OEP in $\bB_i$. Then $P \ceq O_1 \cap O_2$ is an 
orbital in $\bB$, and we claim that it
witnesses the OEP in $\bB$. Let $b_1,b_2$ be elements of $\bB$.
Choose $c_i$ such that $(b_1,c_i) \in O_i$ and $(b_2,c_i) \in O_i$ for $i=1$ and $i=2$. Because $c_1,c_2$ will be distinct from $b_1, b_2$, Lemma~\ref{lem:genericPropertiesWithParameters} yields that there exists $c\in \Aut_{b_1, b_2}(\bB_1)c_1 \cap \Aut_{b_1, b_2}(\bB_2)c_2$, which implies 
$(b_1,c) \in P$ and $(b_2,c) \in P$.
\end{proof}


Finally, we prove Theorem~\ref{thm:r2}. \hyperlink{def:J}{Property~$J$} from Definition~\ref{def:J} is needed in order to apply Proposition~\ref{prop:key} twice.

\begin{proof}[Proof of Theorem~\ref{thm:r2}]
If all polymorphisms of $\bB$ are essentially unary then Proposition~\ref{prop:ecsp-hard} shows that $\Csp(\bB)$ is NP-hard.
Otherwise, $\bB$ has a binary essential polymorphism by Lemma~\ref{lem:kara}, because $\bB$ has the OEP by Lemma~\ref{lem:orb-ext} and Lemma~\ref{lem:free-orb-ext}. 
 \hyperlink{def:J}{Property~$J$} states that $\bB_i$, for $i=1$ and $i=2$, is 2-set-transitive, contains a binary relation symbol that denotes $\neq$ and is without algebraicity. This implies that $\bB_1$ and $\bB_2$ satisfy the assumptions of Proposition~\ref{prop:key}.
It follows that $\bB_1$ has a binary injective polymorphism. By Theorem~\ref{thm:hdr},
this shows that $\Th(\bB_1)$ is convex.

By symmetry also $\Th(\bB_2)$ is convex.
Now, the Nelson-Oppen combination procedure
implies that $\Csp(\bB)$ is in P. 
\end{proof}

\section{Examples} \label{sec:examples}

We will now present examples of structures with \hyperlink{def:J}{Property~$J$}.


\begin{expl}\label{ex:mi}
All first-order expansions of $(\Q; \neq, <)$ have \hyperlink{def:J}{Property~$J$}.
It is easy to check that $(\Q; <)$ is 2-set transitive, $\omega$-categorical and without algebraicity. These properties are generally preserved when expanding a structure with first-order definable relations. 
A cross prevention formula for these structures is
\begin{displaymath}
  \phi(x,y,u,v) \cequiv x < v \AND y > u.
\end{displaymath}

Furthermore, all polynomial-time tractable first-order expansions of $(\Q; <)$ are known (see~\cite{tcsps-journal}) and there exist tractable expansions which are not convex. One such structure is $({\mathbb Q}; \neq, \leq, R_{\mi})$ with
\begin{displaymath}
  R_{\mi} \ceq \{(x,y,z) \in {\mathbb Q}^3 \mid x \geq y \vee x > z\}.
\end{displaymath}
\end{expl}

\begin{expl}
  Let $C$ be the binary branching homogeneous $C$-relation on a countably infinite set \cite{BodJonsPham}. The computational complexity of the $\Csp$s of all first-order expansions of this structure has been fully classified in \cite{Phylo-Complexity}. All these first-order expansions (some of which are polynomial-time tractable) have \hyperlink{def:J}{Property~$J$}. Here, a cross prevention formula is $\phi(x,y,u,v) \cequiv\exists a\big( C(x,u; a) \AND C(y,a; v)\big)$.
  
\end{expl}

Furthermore, we can describe a large family of structures with \hyperlink{def:J}{Property~$J$}:
\begin{lemma}\label{lem:generalExplForJ}
Let $\bB$ be a countably infinite homogeneous relational structure
with finitely many relations, each of arity at least 3 where each tuple in every relation of $\bB$ has pairwise distinct entries. Then $\bB$ is 2-set-transitive. 
\end{lemma}

\begin{proof}
  As $\bB$ is homogeneous with finite relational signature, $\bB$ must be $\omega$-categorical.
  For 2-set-transitivity, consider any four elements $a,b,c,d \in B$ such that $a\neq b$ and $c\neq d$. Because no relation from the signature holds on $a,b$ or $c,d$,
there is an isomorphism between the substructure induced by $\set{a,b}$ in $\bB$ and the substructure induced by $\set{c,d}$ in $\bB$ that maps $(a, b)$ to $(c, d)$. Due to homogeneity of $\bB$, this isomorphism extends to an automorphism of $\bB$. Therefore, $\bB$ is 2-set-transitive.
\end{proof}


\begin{expl}
  An example for a structure $\bB$ satisfying the conditions of Lemma~\ref{lem:generalExplForJ} is the \Fresse -limit of all finite 3-uniform hypergraphs which do not embed a tetrahedron (a surface of the tetrahedron corresponds to a hyperedge). Let $\bB$ be this limit structure with its ternary edge relation $E$ and add $\neq$ to the signature. Then $\bB$ is 2-set-transitive by Lemma~\ref{lem:generalExplForJ}.    
Furthermore, a cross prevention formula for $\bB$ is
  \begin{displaymath}
    \phi(x,y,u,v) \cequiv \exists a,b \big( E(x,a,b) \AND E(u,a,b) \AND E(x,u,a) \AND E(y,v,b) \big).
  \end{displaymath}
  This formula describes a tetrahedron between $a,b,x,u$, except that the edge $E(x,u,b)$ is replaced by $E(y,v,b)$. It is easy to check that both $\phi(x,y,u,v) \AND x=y$ and $\phi(x,y,u,v) \AND u=v$ allow solutions $s_1, s_2$ respectively, such that $s_1(x,u,v)$ and $s_2(x,y,u)$ are injective  but $\phi(x,y,u,v) \AND x=y \AND u=v$ is unsatisfiable. Furthermore, $\Age(\bB)$ has the strong amalgamation property and therefore, $\bB$ does not have algebraicity and hence it has \hyperlink{def:J}{Property~$J$}  by Proposition~\ref{prop:noAlgStrongAmg}.
  
The example can be generalized to work with any complete 3-uniform hypergraph on $n$ vertices for $n\geq 4$ instead of a tetrahedron (i.e., $n=4$).
\end{expl}




To give examples for Theorem~\ref{thm:r2} that are not covered by the Nelson-Oppen combination criterion we need structures with \hyperlink{def:J}{Property~$J$}, with polynomial-time tractable CSP, and without convexity. Here is one such structure and an ostensibly harmless combination partner: 

\begin{expl}\label{expl:mi-comb}
Let $\bB_1$ be the relational structure $({\mathbb Q};\neq, <,R_{\mi})$ 
where $R_{\mi}$ is defined as in Example~\ref{ex:mi}.      
Let $\bB_2 \ceq ({\mathbb Q}; \prec)$   
where $\prec$ also denotes the strict order of the rationals 
(we chose a different symbol than $<$ to make the signatures disjoint). 
It is 
easy to see that $\bB_1$ and $\bB_2$ satisfy
the assumptions of Proposition~\ref{prop:r1}, 
so they have a generic combination $\bB$.

From Example~\ref{ex:mi} we know that $\bB_1$ has \hyperlink{def:J}{Property~$J$} and both have polynomial-time tractable CSPs.
Notice that $R_{\mi}$ prevents binary injective polymorphisms, i.e., it ensures that $\bB_1$ does not have a convex theory (see Theorem~\ref{thm:hdr}). To apply Theorem~\ref{thm:r2} we need a symbol for the disequality relation in $\bB_2$. Let $\not \approx$ be another symbol for the disequality relation on $\bB_2$. Then $\bB_1 \ast (\Q, \not \approx, \prec)$ has the same CSP as $\bB_1 \ast \bB_2$, because $\not \approx$ and $\neq$ will denote the same relation in $\bB_1 \ast (\Q, \not \approx, \prec)$. Hence, Theorem~\ref{thm:r2} implies that the CSP of $\bB_1 \ast \bB_2$ is NP-complete (and we invite the reader to find an NP-hardness proof without using our theorem).

This example shows how expansion of the signature, application of Theorem~\ref{thm:r2}, and subsequent reduction of the signature can extend the scope of Theorem~\ref{thm:r2}.
\end{expl}

Concluding this section, we would like to state the following corollary.
\begin{corollary}\label{cor:tractCombQ}
Let $\bB_1$ and $\bB_2$ be 
two first-order expansions of $({\mathbb Q};<,\neq)$ with finite relational signatures. Rename the relations of $\bB_1$ and $\bB_2$ so that $\bB_1$ and $\bB_2$ have disjoint signatures. 
Then $\Csp(\Th(\bB_1) \cup \Th(\bB_2))$ is in P if 
$\Csp(\bB_1)$ and $\Csp(\bB_2)$ are in P and if both $\Th(\bB_1)$ and $\Th(\bB_2)$ are convex.
Otherwise, $\Csp(\Th(\bB_1) \cup \Th(\bB_2))$ is NP-hard. 
\end{corollary}
This follows from Proposition~\ref{prop:r1} which characterises the existence of a generic combination of $T_1$ and $T_2$, and from Theorem~\ref{thm:r2} which classifies the computational complexity of the generic combination.

\section{Difficulties for a General Complexity Classification}
\label{sect:non-class}
Let $T_1$ and $T_2$ be $\omega$-categorical theories with disjoint finite relational signatures such that
$\Csp(T_1)$ is in P and $\Csp(T_2)$ is in P. 
The results in this section suggest that 
 in general we cannot hope to get a classification 
of the complexity of $\Csp(T_1 \cup T_2)$. We use the result from~\cite{BodirskyGrohe}
that there are homogeneous directed graphs $\bB$ such that $\Csp(\bB)$ is undecidable. 
There are even homogeneous directed graphs
$\bB$ 
such that $\Csp(\bB)$ is \emph{coNP-intermediate}, i.e.,  
in coNP, but neither coNP-hard nor in P~\cite{BodirskyGrohe} (unless P = coNP).
All of the homogeneous graphs $\bB$ used in~\cite{BodirskyGrohe} can be described
by specifying a set of finite tournaments $\mathcal T$. Let $\mathcal C$ be the class of all finite directed loopless graphs $\bA$ such that no tournament from $\mathcal T$ embeds into $\bA$. It can be checked
that $\mathcal C$ is a strong amalgamation class;
the \Fresse-limits of those classes are called the
\emph{Henson digraphs}.

\begin{proposition}\label{prop:non-class}
For every Henson digraph $\bB$  
there exist $\omega$-categorical convex theories $T_1$ and $T_2$ with disjoint finite relational signatures such that 
$\Csp(T_1)$ is in P, $\Csp(T_2)$ is in P, and 
$\Csp(T_1 \cup T_2)$ is polynomial-time Turing equivalent to $\Csp(\bB)$. 
\end{proposition}
\begin{proof}
Let $\bB_1$ be the structure obtained from $\bB$ by adding a new element $a$ and adding the tuple $(a,a)$ to the edge relation $E$ of $\bB$.
Clearly, $\bB_1$ is $\omega$-categorical and $\Csp(\bB_1)$ is in P. 
Let $\bB_2$ be $(\N, \neq)$. Let $T_i$ be the theory of $\bB_i$, for $i=1$ and $i=2$. We first present a polynomial-time reduction from $\Csp(\bB)$ to $\Csp(T_1 \cup T_2)$.

First observe that $x\mapsto a$ is the only solution to the instance $\set{ E(x,x)}$. Hence, given an instance $S$ of $\Csp(\bB)$ with variables $x_1, \dots, x_n$, we can give the following to an algorithm for $\Csp(T_1\cup T_2)$:
\begin{displaymath}
 S^* \ceq  \set{E(x_0,x_0),\, x_0\neq x_1, \dotsc, x_0\neq x_n} \cup S
\end{displaymath}
If there is a solution $s$ for $S^*$ in $\Csp(T_1\cup T_2)$, then the values $s(x_1),\dotsc, s(x_n)$ are distinct from $a$ and therefore $s|_{\set{x_1, \dotsc, x_n}}$ is a solution for $S$ in $\bB$. Likewise, if there is a solution $s$ for $S$ in $\bB$, then $s$ and $x_0 \mapsto a$ gives a solution for $S^*$ in $\Csp(T_1\cup T_2)$.

Next, we give a polynomial-time reduction from $\Csp(T_1 \cup T_2)$ to $\Csp(\bB)$. Consider an instance $S$ of $\Csp(T_1 \cup T_2)$ and assume without loss of generality that there is no conjunct of the form $x=y$ (substitute $y$ with $x$ otherwise and revert the substitution after the algorithm otherwise). Now we consider the directed graph $(V; R)$ where $V$ is the set of variables occurring in the instance $S$ and $R$ is the set of pairs $(x_i,\, x_j)$ such that $E(x_i,x_j)$ is in $S$.

Let $C_1,\dots, C_n$ be the connected components of $(V; R)$ where two nodes count as connected if there is some edge between them, regardless of its direction. Furthermore define $D\ceq \set{(x_i, x_j) \mid x_i\neq x_j \text{ is a constraint in } S}$ and label all $C_k$ where $\set{ E(x_i,x_j)\in S \mid x_i,x_j\in C_k}$ has no solution in $\bB$ with $a$. The algorithm proceeds as follows: If there exists $(x_i,x_j)\in D$ such that both $x_i$ and $x_j$ are in components labeled with $a$, then reject. Accept the instance otherwise.

To see that this is correct, notice that a connected component labeled $a$ can only be satisfied by setting all variables in that component to $a$. Therefore, all rejected cases must be unsatisfiable.

Now, consider an instance $S'$ of $\Csp(\bB)$ obtained by this reduction. There is a solution to the instance iff no tournament $G_F\in \mathcal{T}$ can be embedded in the graph $G_{S'}$ given by the edge constraints in $S'$. If there is an embedding $f\colon G_F \hookrightarrow G_{S'}$ for some  $G_F\in \mathcal{T}$, then either $E(x_i,x_j)$ or $E(x_j,x_i)$ is in $S'$ for all variables $x_i,x_j$ in the image of $f$. Hence, if we started with an additional constraint $x_i\neq x_j$ (for two variables $x_i,x_j$ with $i\neq j$), removing this constraint would never make an otherwise unsatisfiable instance satisfiable, because the same $G_F$ would still be embeddable. Therefore, any satisfiable instance of $\Csp(\bB)$ is also satisfiable when arbitrary additional $\neq$-constraints between different variables are enforced.
Hence, to construct a solution for the non-rejected cases of $S$, we take an injective solution $s$ for the union of all unlabeled $C_k$ of $S'$ such that $a$ is not in the image of $s$ and map all other variables to $a$. It is easy to check that this satisfies all constraints in $S$.
\end{proof}

We mention that another example of two theories such that  $\Csp(T_1)$ and $\Csp(T_2)$ are decidable but $\Csp(T_1 \cup T_2)$ is not can be found in \cite{BonacinaGhilardiUA2006_DecideNORewrite}.

\medskip 
{\bf Acknowledgements.} We would like to thank Marcello Mamino for many helpful discussions about this material. Furthermore, we thank the anonymous reviewers for their helpful comments.

\bibliographystyle{alpha}

\bibliography{local}

\end{document}